\documentclass[12pt]{article}
\usepackage{latexsym,amssymb,amsmath,amsfonts,amsthm}
\usepackage{natbib}
\usepackage{color}
\usepackage{tikz}
\usepackage{comment}
\usepackage{hyperref,graphicx}
\usepackage{multirow}
\usepackage{sidecap,ulem}

\newtheorem{theorem}{Theorem}
\newtheorem{definition}{Definition}
\newtheorem{lemma}{Lemma}

\def\beq{ \begin{equation} }
\def\eeq{ \end{equation} }

\def\square{\vcenter{\vbox{\hrule height .4pt
  \hbox{\vrule width .4pt height 5pt \kern 5pt
        \vrule width .4pt} \hrule height .4pt}}}

\def\E{\mathbb{E}}
\def\P{\mathbb{P}}
\def\R{\mathbb{R}}

\def\Z{\mathbb{Z}}

\def\T{\mathcal{T}}

\title{An Adjacent-Swap Markov Chain on Coalescent Trees}
\author{Mackenzie Simper, Julia A. Palacios}
\date{}

\begin{document}

\maketitle

\begin{abstract}
The standard coalescent is widely used in evolutionary biology and population genetics to model the ancestral history of a sample of molecular sequences as a rooted and ranked binary tree. In this paper, we present a representation of the space of ranked trees as a space of constrained ordered matched pairs. We use this representation to define ergodic Markov chains on labeled and unlabeled ranked tree shapes analogously to transposition chains on the space of permutations. We show that an adjacent-swap chain on labeled and unlabeled ranked tree shapes has mixing time at least of order $n^3$, and at most of order $n^{4}$. Bayesian inference methods rely on 
Markov chain Monte Carlo methods on the space of trees. Thus, it is important to define good Markov chains which are easy to simulate and for which rates of convergence can be studied.
\end{abstract}

\section{Introduction}

The standard coalescent \citep{Kingman:1982uj} is often used in evolutionary biology and population genetics to model the set of ancestral relationships among $n$ individuals as a rooted and timed binary tree called a genealogy \citep{wakeley_coalescent_2008}. In this manuscript, we are concerned with  discrete tree topologies only and assume a unit interval length between consecutive coalescing (branching) events.  

A labeled \textbf{ranked tree shape} $T^{L} \in \mathcal{T}^{L}_{n}$ with $n$ leaves is a rooted labeled and ranked binary tree. The leaves are labeled by the set $\{\ell_{1},\ell_{2},\ldots,\ell_{n}\}$ and internal nodes are labeled  $1,\ldots,n-1$ in increasing order, with label $n-1$ at the root (Figure \ref{fig:fig1}(A)). The cardinality of $\mathcal{T}^{L}_n$ is given by
\begin{equation}
|\mathcal{T}^{L}_{n}|=\frac{n!(n-1)!}{2^{n-1}}.
\end{equation}


In many applications, the quantity of interest is the overall shape of the tree \citep{Kirkpatrick1993,Frost2013,Maliet2018}; the specific labels of the samples are un-important. In fact, a key feature of simple coalescent models, such as the Kingman coalescent, is the assumption that the leaves are \textit{exchangeable}. The Tajima coalescent, a lumping of the Kingman coalescent that ignores leaf labels, 
was recently proposed for inferring past population size trajectories from binary molecular data \citep{Palacios2019}. This motivates the study of \textbf{unlabeled ranked tree shapes} with $n$ leaves $\mathcal{T}_n$ (Figure \ref{fig:fig1}(B)).  The appealing feature of modeling unlabeled ranked tree shapes is the reduction in the cardinality of the state space. The cardinality of $\mathcal{T}_{n}$ is given by the ($n-1$)-th Euler zig-zag number $E_{n-1} \ll |T^{L}_{n}|$ \citep{kuznetsov1994increasing}.


\begin{figure}
\includegraphics[width=0.5\textwidth]{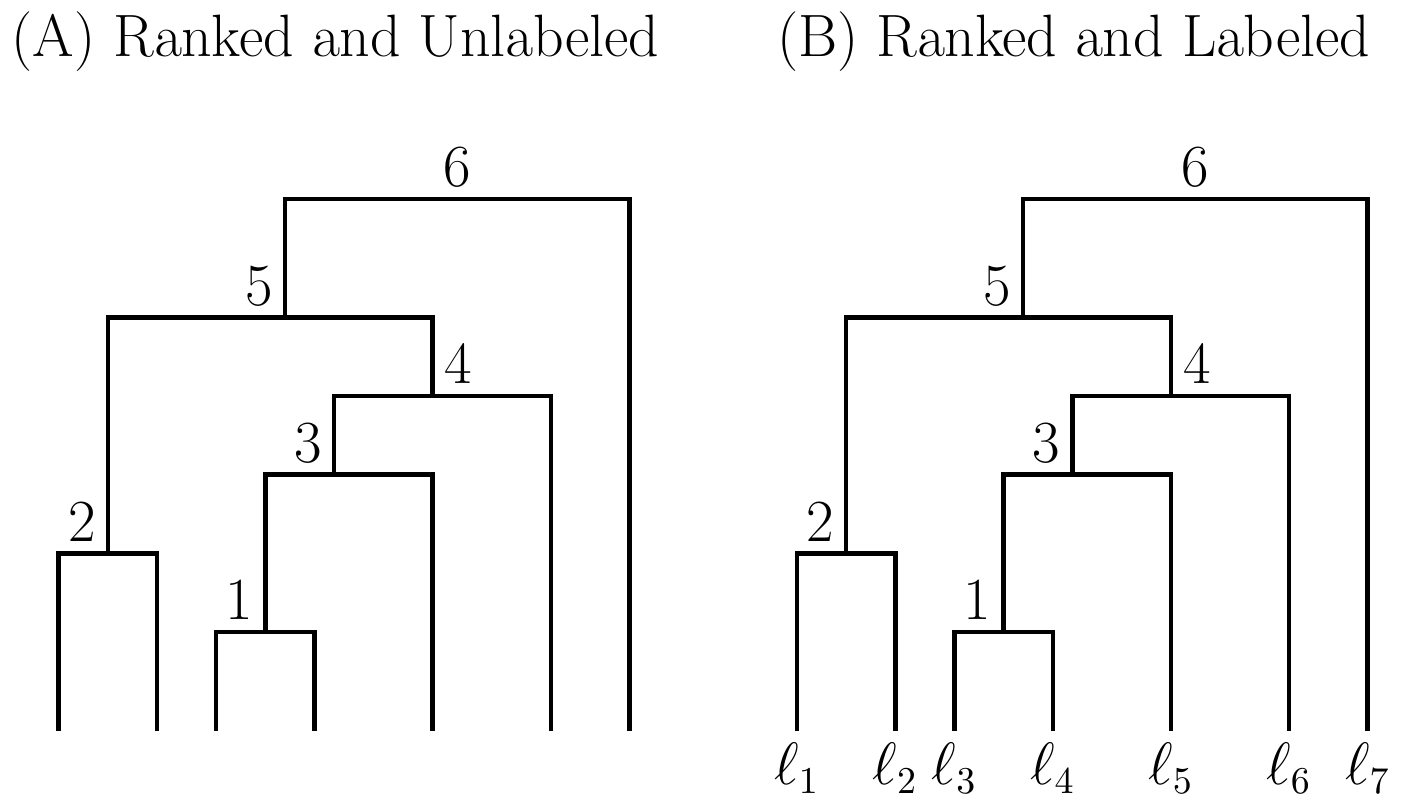}
\centering
\caption{\textbf{Examples of coalescent tree topologies}. A ranked and unlabeled tree shape $T$ has internal nodes labeled by their ranking from leaves to root and unlabeled leaves. (A) has ordered matched pairs: $(0,0)^{1},(0,0)^{2},(0,1)^{3},(0,3)^{4},(2,4)^{5},(0,5)^{6}$. A ranked and labeled tree $T^{L}$ has unique leaf and internal node labels. (B) has ordered matched pairs: $(\ell_{3},\ell_{4})^{1},(\ell_{1},\ell_{2})^{2},(\ell_{5},1)^{3},(\ell_{6},3)^{4},(2,4)^{5},(\ell_{7},5)^{6}$.}
\label{fig:fig1}
\end{figure}

Other types of trees are used in the context of evolutionary models, such as cladeograms (or phylogenies). Cladeograms are rooted or unrooted binary trees with labeled leaves and no ranking of internal nodes. In \cite{diaconis1998matchings}, 
the space of cladeograms with $l$ leaves is shown to be in bijection with the space of perfect matchings of $2n$ labels, with $n = l - 1$. The bijection with perfect matchings is used to define a Markov chain on the space of cladeograms. Using the perfect matching perspective, the chain is analogous to a \textit{random transpositions} chain on the space $S_{2n}$ of permutations of $2n$ elements. This representation allows the use of known results for the random transpositions chain to be applied to the chain on trees to determine sharp bounds on rates of convergence \cite{diaconis2002random}. Motivated by this result, we propose a new representation of labeled and unlabeled ranked tree shapes as a type of matching. We use this representation to define an adjacent-swap chain on the spaces of labeled and unlabeled ranked tree shapes, and to study their mixing time. It is important to study the mixing time for chains on ranked trees because these results have implications for the design of MCMC algorithms to infer evolutionary parameters from molecular data in a Bayesian setting \citep{drummond2012bayesian}.\\
%
%

In the rest of this section we define a bijective representation of ranked trees as constrained ordered matchings and define the adjacent-swap chain. We then state our main results concerning mixing times and discuss related work. 
In Section \ref{sec: prelim}, we state known results on Markov chains that will be used in the following sections. In section \ref{sec: chain} we state important properties of the proposed adjacent-swap chain. The lower bound for the mixing time of the adjacent-swap chains is proven in section \ref{sec: lowerBound} and it is obtained by finding a specific function that gives a lower bound for the relaxation time of the chains.
In section \ref{sec: coupling}, we prove the upper bounds on the mixing time of the adjacent-swap chains via a coupling argument. In the discussion section we mention future directions and other related chains on labeled and un-labeled coalescent trees which could be defined using the matching representations, as well as discuss implications for MCMC methods applied in practice.

\subsection{Matching representations}

We first establish some definitions. Let $m$ be an integer and $S$ a multiset of size $2m$. A  \textit{matching} of $S$ objects is a partition of the elements into $m$ disjoint subsets of size two. An \textit{ordered matching} is a matching with a linear ordering (ranking) of the $m$ pairs. 

\begin{definition}
 The space $COM_m(S)$ (Constrained Ordered Matchings) is the set of all ordered matchings of $S$ with pairs labeled $\textbf{p}_{1}, \textbf{p}_2, \dots, \textbf{p}_{m}$, which satisfy, for $k = 1, \dots, m$ and $a, b \in S$,
\begin{equation} \label{eqn: constraint}
\textbf{p}_k = (a, b)^{k} \implies a, b < k \text{ if } a,b \in \Z \text{, or } a,b \in \{\ell_{1},\ldots,\ell_{n}\}
\end{equation}
\end{definition}

To represent a labeled ranked tree shape as a constrained ordered matching, we take $S$ to be the set of leaf and internal node labels, 
and each pair in the matching represents a coalescence in the tree, with $\textbf{p}_{k}$ representing the $k$-th coalescence event. The condition \eqref{eqn: constraint} then simply ensures that no interior node can be merged before it has been introduced. A leaf can merge at any point and thus there is no constraint on a leaf's position in the matching.  

\begin{lemma}
With $S = \{\ell_{1},\ell_{2}, \dots, \ell_{n}, 1, 2, \dots, n - 2 \}$,
the space $COM_{n-1}(S)$ is in bijection with $\T_n^L$. With multiset $S = \{0, 0, \dots, 0, 1, 2, \dots, n - 2\}$, with $0$ repeated $n$ times, the space $COM_{n-1}(S)$ is in bijection with with $\T_n$.
\end{lemma}

\begin{proof}
First assume that $S = \{\ell_{1},\ell_{2}, \dots, \ell_{n}, 1, 2, \dots, n - 2 \}$. The first matched pair can be formed by any two leaves in $\binom{n}{2}$ possible ways. The second pair can be formed by any of the $n-2$ remaining leaves or $1$, the label of the first matched pair, in $\binom{n-1}{2}$ ways; in general, the $k$-th matched pair can be formed in $\binom{n-k+1}{2}$ ways, because there are $n - k + 1$ unique labels that could be matched at that time. Thus, $|COM_{n-1}(S)|=\prod^{n}_{i=2}\binom{i}{2}=|\T^{L}_{n}|$ and for every element $T^{L} \in\mathcal{T}^{L}_{n}$ there is a one-to-one mapping between every coalescence event in $T^{L}$ and every matched pair in one element of $COM_{n-1}(S)$. These two facts imply there is a one-to-one mapping between $\mathcal{T}^{L}_{n}$ and $COM_{n-1}(S)$.  

With the leaf labels $\ell_1, \dots, \ell_n$ replaced with the repeated element $0$, the space of matchings is in bijection with $\T_n$ due to the fact that the space of $\T_n$ is equivalent to the space of $\T_n^L$ with the leaf labels removed. That is, there is a surjection $\T_n^L \to \T_n$ equivalent to the surjection $COM_{n-1}(S_1) \to COM_{n-1}(S_2)$, with $S_1 =  \{\ell_{1},\ell_{2}, \dots, \ell_{n}, 1, 2, \dots, n - 2 \}$ and $S_2 = \{0, 0, \dots, 0, 1, 2, \dots, n - 2 \}$.  The bijective matching representation of the ranked trees $\T_n$ and $\T_n^L$ of Figure \ref{fig:fig1} are shown in the legend of Figure \ref{fig:fig1}.
\end{proof}

\subsection{Markov chain and main results}

\begin{definition}
 The \textit{adjacent-swap Markov chain} on the space $COM_{n-1}(S)$
is defined by the following update move:
 \begin{enumerate}
     \item Pick an index $k \in \{1, \dots, n-2 \}$ uniformly at random.
     \item Pick labels $l_1$ and $l_2$ uniformly from the pairs $\mathbf{p}_{k}$ and $\mathbf{p}_{k+1}$ respectively.
     \begin{enumerate}
         \item If swapping the positions of $l_1, l_2$ does not violate constraint \eqref{eqn: constraint}, make the swap.
         \item Otherwise, remain in the current state. 
     \end{enumerate}
 \end{enumerate}
\end{definition}

For the sets $S$ which give the bijections with $\T_n$ and $\T_n^L$, this Markov chain is connected and reversible. In the case of $\T_n^L$ it is also symmetric and hence, it has a uniform stationary distribution. For $\T_n$, the chain is not symmetric; the stationary distribution is the \textit{Tajima distribution} on the space of ranked un-labeled trees. These facts, as well as the relationship between the two chains, will be proved in a later section.

\begin{figure}
\includegraphics[width=\textwidth]{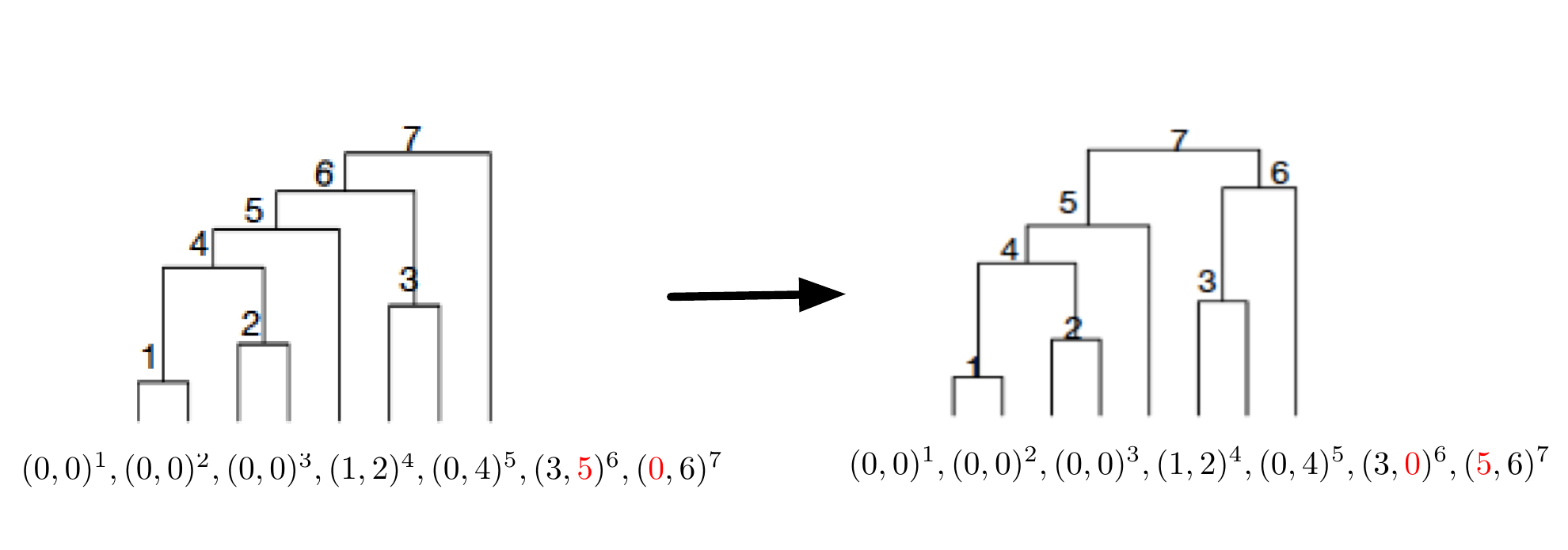}
\centering
\caption{\textbf{Example of an adjacent-swap move on ranked tree shapes}. The pair $k=6$ is selected uniformly at random among pairs $2,\ldots,6$. We then select uniformly at random one element from pair $6$ and one element from pair $7$. The only allowable swaps are $5$ and $0$ or $3$ and $0$. The new state (right tree) is obtained by swapping $5$ and $0$ from pairs $6$ and $7$ respectively in the left tree.}
\label{fig:fig2}
\end{figure}

We are interested in the convergence rates of this Markov chain and how this rate compares depending on the space $\mathcal{T}_n$ or $\mathcal{T}_n^L$. The measure that we study for convergence rate is the \textit{mixing time}. That is, for a Markov chain on a space $\Omega$ with transition probability $P$ and stationary distribution $\pi$, the mixing time is defined as
\[
t_{mix} = \sup_{x_0 \in \Omega} \inf \{ t > 0 : \| P^t(x_0, \cdot) - \pi(\cdot) \|_{TV} < 1/4 \}.
\]
To state the result, let $t_n$ and $t_n^L$ be the mixing times for the chains on $\T_n$ and $\T_n^L$ respectively. Our main result is:

\begin{theorem} \label{thm: main}
There exists constants $C_1, C_2, C_3$ such that
\[
C_1 \cdot n^3 \le t_n \le C_2 \cdot n^4,
\]
and
\[
C_1 \cdot n^3 \le t_n^L \le C_3 \cdot n^4.
\]
\end{theorem}

The lower bound of $n^3$ comes from the \textit{relaxation time}, and involves a standard trick of finding a specific function for which the variance under the stationary distribution can be bounded. The chain on $\T_n$ is a certain type of ``lumping'' of the chain on $\T_n^L$, and so the same lower bound applies to both spaces. The upper bounds are obtained using a coupling argument. We conjecture that the mixing time $t_n$ is indeed smaller than the time $t_n^L$, though it is not evident whether this is by a significant order, or just a constant factor. 


\subsection{Related work}

Mixing of chains on other tree spaces have been previously studied, in particular, the mixing of Markov chains on the space of cladeograms.  An $n$-leaf \textit{cladeogram} is a rooted or unrooted tree with $n$ labeled leaves and unlabeled internal nodes of degree 3. The main difference between cladeograms and the tree topologies studied in this manuscript, is that cladeograms do not rank internal branching events. The cardinality of the space of cladeograms of $n$ leaves is $(2(n-1))!/2^{n-1}(n-1)!$, smaller than the cardinality of the space of ranked and labeled trees, but larger than the cardinality of the space of ranked unlabeled trees. Cladeograms are fundamental objects in phylogenetics to model ancestral relationships at the species level, while ranked tree shapes are fundamental objects in population genetics and phylodynamics of infectious diseases. Ranked tree shapes are used to model ancestral relationships of a sample of individuals from a single population of the same species \citep{felsenstein2004inferring}. \\

\noindent\textbf{Markov chains on cladeograms.} Aldous \cite{aldous2000mixing} studied a chain on unrooted cladeograms that removes a leaf chosen uniformly at random from the current cladeogram and reattaches it to a random edge. Using coupling methods, Aldous \cite{aldous2000mixing} showed the relaxation time for this chain is $\tau_{rel} \leq c_{2}n^{3}$. The same chain was later analyzed by Schweinsberg \citep{schweinsberg2002n2} who showed that $\tau_{rel} = O(n^{2})$ using a modified method of distinguished paths.
In \cite{diaconis1998matchings}, the authors show that the space of rooted cladeograms with $n$ leaves is in bijection with the space of perfect matchings in the complete graph on $2(n-1)$ vertices. This bijection was later used to define a  Markov chain by randomly choosing two pairs and transposing two randomly selected entries of each pair \cite{diaconis2002random}. The authors showed that this random transpositions chain mixed after $\frac{1}{2}n \log n$ steps. In \cite{SHK2014}, the authors study the nearest neighbor interchange (NNI) chain and a subtree prune and regraft (SPR) chain on the space of rooted cladeograms. The authors showed that the upper bounds on the relaxation time of the SPR and the NNI chains are $O(n^{5/2})$ and $O(n^{4})$ respectively. 

Other related work include the study of mixing times of chains whose  stationary distribution is the posterior distribution over cladeograms \citep{mossel2006limitations,vstefankovivc2011fast}. In 
\citep{mossel2006limitations}, the authors show exponential mixing times under a misspecified model in which data is generated from a mixture of cladeograms. 


\paragraph{Markov chains on related spaces.} The adjacent swap Markov chain on the space of ranked trees studied in this manuscript is closely related to the adjacent transpositions chain on the space $S_n$ of permutations. 
As it is often done, one can think of an element $\sigma \in S_n$ as an ordering of a deck of cards with unique labels $1, \dots, n$. In \cite[Example 4.10]{aldous1983random}, the chain is defined by picking at random two adjacent cards from the deck; with probability 1/2, the cards are swapped and with probability 1/2 nothing happens. 
The chain is symmetric with uniform stationary distribution. Aldous \citep{aldous1983random} showed that the mixing time of this chain has a lower bound of $c_{1}n^{3}$ and an upper bound of $c_{2}n^3 \log(n)$ on the mixing time.
In \cite{wilson2004mixing}, the now-famous ``Wilson's method'' was introduced and applied to the adjacent transpositions chain as an example to improve the lower bound to order $n^3 log(n)$ with the explicit constant $(1/\pi^2) n \log(n)$, as well as an upper bound of $(2/\pi^2) n \log(n)$. Finally, in \cite{lacoin2016mixing} ,the upper bound was improved to $(1/\pi^2) n \log(n)$, so that the upper and lower bounds matched; in addition, the chain was proven to follow the \textit{cut-off} phenomenon. 
Durrett generalized the adjacent transpositions chain  to the $L$-reversal chain, introduced as a model for the evolution of DNA in a chromosome \cite{durrett2003shuffling}. 


To the best of our knowledge, bounds on the mixing time of chains on  ranked tree shapes have not been studied before \citep{misra2011optimization}. 


\section{Preliminaries} \label{sec: prelim}

In this section, we review some results of Markov chains theory that will be used to bound the mixing time of the adjacent swap chains on $\T^{L}_{n}$ and $\T_{n}$. We use the coupling method to find the upper bounds. A \textit{coupling} of Markov chains with transition matrix $P$ is a process $(X_t, Y_t)_{t \ge 0}$ such that both $(X_t)$ and $(Y_t)$ are marginally Markov chains with transition matrix $P$. The goal is to define a coupling of two chains from different starting distributions such that the two chains will quickly reach the same state, and once this occurs the two chains stay matched forever. This \textit{coupling time} gives an upper bound on the mixing time, see Chapter 5 from \cite{LevinPeresWilmer2006} for more details.

\begin{theorem}[Theorem 5.4 \citep{LevinPeresWilmer2006}] \label{thm: couplingThm}
Suppose that for each pair of states $x,y \in \mathcal{X}$ there is a coupling $(X_{t},Y_{t})$ with $X_{0}=x$ and $Y_{0}=y$. For each such coupling, let $\tau_{\text{couple}}$ be the coalescence time of the chains, i.e. 
\[ \tau_{\text{couple}}:=\min \{t: X_{s}=Y_{s} \text{ for all } s \geq t\}.
\]
Then
\[t_{mix} \leq 4 \max_{x,y}E_{x,y}(\tau_{\text{couple}})\]
\end{theorem}

%

To find a lower bound on the mixing time, we bound the relaxation time and use the following result:
\begin{theorem}[Theorem 12.5 \citep{LevinPeresWilmer2006}]
Suppose $P$ is the transition matrix of an irreducible, aperiodic and reversible Markov chain. Then 
\begin{equation}
t_{mix} \geq (\tau_{n}-1)\log(2)
\end{equation}
\end{theorem}

The relaxation time is defined as the inverse of the spectral gap: $\tau_{n}=1/\gamma$, where $\gamma=1-|\lambda_{n,2}|$, one minus the absolute value of the second-largest eigenvalue of the transition matrix of the chain. We will bound  the relaxation time using the following variational characterization:

\begin{theorem}[Lemma 13.7 \citep{LevinPeresWilmer2006}] Let P be the transition matrix for a reversible Markov chain. The spectral gap $\gamma$ satisfies
\[
\gamma = \min _{f:\mathcal{X} \rightarrow \R, \text{Var}_{\pi}(f) \neq 0} \frac{\mathcal{E}(f)}{\text{Var}_{\pi}(f)}
\]
where
\[\mathcal{E}(f):=\frac{1}{2}\sum_{x,y \in \mathcal{X}}[f(x)-f(y)]^{2}\pi(x)P(x,y)
\]
\end{theorem}

\section{The adjacent-swap chain on ranked tree shapes} \label{sec: chain}
%

\begin{lemma}
The adjacent-swap Markov chain on $\T_n^L$ is irreducible, aperiodic and reversible with respect to the uniform stationary distribution $\pi^L(x) := 1/|\T_n^L|$.
\end{lemma}
%

\begin{proof}
This can be seen by noting the transition matrix is symmetric. Let $P^{L}$ be the transition matrix for the adjacent-swap chain on $\T_n^L$, then 
\[
P^L(x, y) = \frac{1}{n-2} \cdot \frac{1}{4},
\]
where $y$ is obtained by swapping two randomly chosen elements from two adjacent pairs in $x$, chosen uniformly at random, e.g. consider the following states at pairs $k$ and  $k + 1$:
\begin{align*}
&x: (a, b)^k, (c, d)^{k + 1} \\
&y: (a, c)^k, (b, d)^{k + 1}.
\end{align*}
In a labeled ranked tree shape $T^{L}_{n}$, all pairs are formed by distinct elements, so $a, b, c, d$ are all unique and the only way to transition from $x$ to $y$ is to swap the labels $b$ and $c$, which happens with probability $1/4(n - 2)$ if $c<k$ or $c \in \{\ell_{1},\ldots,\ell_{n}\}$. 

To see that the chain is irreducible, we note that any label can be moved to any pair to the right, one pair (or step) at a time with positive probability and that any label can be moved to a pair with index $k$ to the left (one pair at a time) as long as the label corresponds to a leaf or an internal node smaller than $k$, satisfying constraint (\ref{eqn: constraint}). To see that the chain is aperiodic, we  note that $P^{L}(x,x)>0$, for all $x\in \T^{L}_{n}$. Independent of the current state, we can pick pair $k=n-2$ and attempt to swap any element of the $n-2$ pair with label $n-2$ from the $n-1$ pair with probability $1/2(n-2)$. Since this move violates constraint \eqref{eqn: constraint}, the move is rejected and the chain remains in the current state.
\end{proof}

We note that the transition matrix $P$ of the adjacent-swap chain on unlabeled ranked tree shapes is not symmetric. For example, consider the transition $x$ to $y$ of the type:
\begin{equation} \label{eq:cherr}
(0, a)^k, (0, b)^{k+1} \to (0, 0)^k, (a, b)^{k+1},
\end{equation}
for labels $a, b \neq 0$ (recall the $0$s represent leaf labels). The probability of this transition is $P(x,y)=1/(n-2) \cdot 1/4$, since once pairs $k$ and $k + 1$ are chosen, there is a $1/4$ chance of choosing  label $a$ from pair $k$ and label $0$ from pair $k+1$. However, the reverse transition:
\[
(0, 0)^k, (a, b)^{k+1} \to (0, a)^k, (0, b)^{k+1}
\]
has probability $1/(n-2) \cdot 1/2$. It turns out that the stationary distribution $\pi$ for the chain is  the \textit{Tajima coalescent} distribution \cite{veber} (also known as Yule distribution). For $T \in \T_n$
\[
\pi(T) = \frac{2^{n - c(T) - 1}}{(n-1)!}, 
\]
where $c(T)$ is the number of cherries of $T$, i.e. the number of pairs of the type $(0,0)$ in the $COM_{n-1}(S)$ representation. Indeed, for transitions $x$ to $y$ of the type of \eqref{eq:cherr},  $y$ has one more cherry than $x$ and $\pi(x)P(x,y)=\pi(y)P(y,x)$; for other types of transitions $P(x,y)>0$ that do not affect the number of cherries, $P(x,y)=P(y,x)$, $\pi(x)=\pi(y)$ and the detailed balance equation is satisfied. 

Another way of proving the Tajima coalescent distribution is the stationary distribution on unlabeled ranked tree shapes is to view the chain as a  \textit{lumping} of the chain on $\T^{L}_{n}$. This perspective also gives us an initial comparison of the relaxation times of the chains. The space $\T_n$ of unlabeled ranked-tree shapes can be considered as a set of equivalency classes of the trees $\T_n^L$. That is, for trees $x, y \in \T_n^L$, define the equivalence relation $x \sim y$ if the trees have the same ranked tree shape.  From the $COM_{n-1}(S)$ perspective with $S = \{\ell_1, \dots, \ell_n, 1, 2, \dots, n - 2\}$, two matchings are equivalent if all internal node labels  $1,\ldots,n-2$ occur in the same pairs.

This equivalence relation induces a partition of $\T_n^L$ using equivalence classes. That is, we can write $\T_n^L$ as the disjoint union of sets $\Omega_1, \dots, \Omega_M$, where $M = |\T_n|$, and all trees in $\Omega_i$ have the same ranked-tree shape. 

\begin{lemma} \label{lem: lumping}
For any $x, x'  \in \T_n^L$, equivalence class $\Omega_i$, and $x \sim x'$, the following relation holds:
\[
P^L(x, \Omega_i) := \sum_{y \in \Omega_i} P^L(x, y) = P^L(x', \Omega_i).
\]
\end{lemma}

\begin{proof}
If $x\in \Omega_{i}$, then the transition from $x$ to $y\in \Omega_{i}$ is of the type:
\begin{equation}
(a, b)^k, (c, d)^{k+1} \to (c, b)^k, (a, d)^{k+1}.
\label{eq:type}
\end{equation}
where $a, c \in \{\ell_1, \ell_2, \dots, \ell_{n} \}$ correspond to leaf labels. Since $x$ and $x'$ differ only at leaf labels, the transition probability of swapping two leaf labels is the same and $P^{L}(x,\Omega_{i})=P^{L}(x',\Omega_{i})$. If $x \not\in \Omega_{i}$, and therefore $x' \not\in \Omega_{i}$, then the transition from $x$ to $y$ is of the type:
\[
(a, b)^k, (c, d)^{k+1} \to (c, b)^k, (a, d)^{k+1}.
\]
where $a$ and $c$ are not both in $\{\ell_1, \ell_2, \dots, \ell_{n} \}$. Since $x$ and $x'$ differ only at leaf labels, again $P^{L}(x,\Omega_{i})=P^{L}(x',\Omega_{i})$.
%
\end{proof}

In words, Lemma \ref{lem: lumping} says that probability of transitioning to a different ranked-tree shape is independent of the leaf configuration of the current tree. A well-known result (e.g.\ Lemma 2.5 in \cite{LevinPeresWilmer2006}) is that the induced chain on the space of equivalence classes defined by $\tilde{P}([x], [y]) := P(x, [y])$ is a Markov chain; then observe that the induced chain $\tilde{P}$ is equivalent to adjacent-swap Markov chain $P$ on the space $\T_n$. This allows a comparison of the spectral gaps of the two chains. Lemma 12.9 in \citep{LevinPeresWilmer2006} states the eigenvalues of the transition matrix of the lumped chain are eigenvalues of the transition matrix of the full chain, hence we have the following lemma.

\begin{lemma} \label{lem: lumping2}
Let $\gamma, \gamma^L$ be the spectral gaps of the chains $P, P^L$ respectively. Then, $\gamma^L \leq \gamma$.
\end{lemma}

%

From Lemma \ref{lem: lumping} we can immediately see that stationary distribution for $P$ is indeed the Tajima coalescent distribution. Suppose $T \in \T_n^U$ corresponds to the equivalence class $\Omega_i \subset \T_n^L$. Then,
\[
\pi(T) = \sum_{y \in \Omega_i} \pi^L(y) = \frac{|\Omega_i|}{|\T_n^L|}.
\]
A uniformly sampled labeled ranked tree shape with the leaf-labels erased gives an unlabeled ranked tree shape according to the Tajima distribution \citep{veber}. This implies that $|\Omega_i|/|\T_n^L|$ is equal to the Tajima distribution for the ranked  tree shape corresponding to $\Omega_i$. 



\section{Lower bound} \label{sec: lowerBound}

To prove the lower bounds in Theorem \ref{thm: main}, we use the variational characterization of the relaxation time:


\begin{equation} \label{eqn: relTime}
\tau_{n} = \sup_{f: \Omega \to \R} \frac{Var_\pi(f)}{\mathcal{E}(f)}.
\end{equation}
We will use the \textit{internal tree length} as a function $\varphi(x): \T_{n} \rightarrow \R$ to find a lower bound for $\tau_n$. Since the internal tree length is independent of the leaf labels, we get the same lower mixing time bound for the adjacent-swap chains on $\T_{n}$ and $\T^{L}_{n}$. For ease of exposition we will focus the following discussion on $\T_{n}$.\\


We will define the internal tree length function on the space $\mathcal{T}_n$ using the correspondence with $COM_{n-1}(S)$ and assuming unit length between consecutive coalescence events. For a label $j \in \{1, \dots, n - 2 \}$, let $I_{x}(j)$ denote the pair index $k$ such that $j \in \textbf{p}_k$. In the example of Figure \ref{fig:fig3}, $I_{x}(2) = 5$. Note that from the constraints (Eq.~\eqref{eqn: constraint}), we have that
\[
I_{x}(j) \in \{j+1, \dots, n - 1 \}.
\]
We now define the internal tree length function on $\mathcal{T}_n$ as follows:
\begin{equation}
\varphi(x) = \sum_{j = 1}^{n-2} (I_{x}(j) - j).
\end{equation}

As an example consider the  ranked unlabeled tree shape with $n = 7$ depicted in Figure \ref{fig:fig3}. Note that the function is minimized for the ``caterpillar tree'': $(0,0)^{1},(0,1)^{2},(0,2)^{3},\ldots,(0,n-2)^{n-1}$, where $I_{x}(j) = j + 1$ for all $j$, giving $\varphi(x) = n - 2$. 

\begin{SCfigure}[3.0][!hb]
\includegraphics[width=0.2\textwidth]{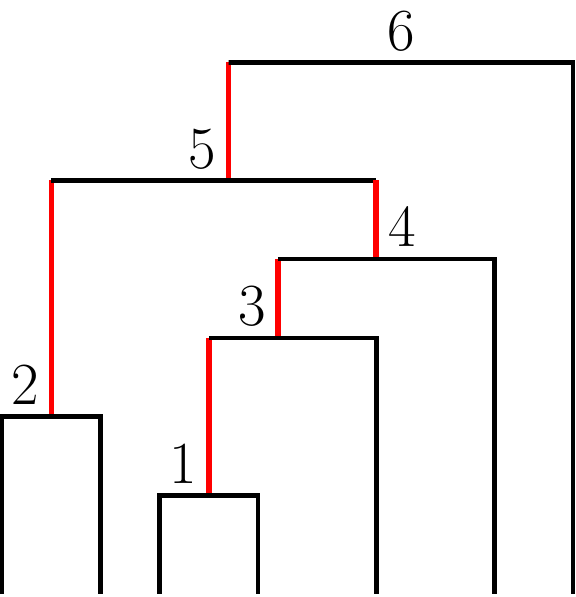}
\caption{\textbf{Example of internal tree length calculation}. A ranked tree shape $T$ with ordered matched pairs: $x=(0,0)^{1},(0,0)^{2},(0,1)^{3},(0,3)^{4},(2,4)^{5},(0,5)^{6}$. The corresponding internal tree length is the sum of the lengths of the red branches: $\varphi(x)=(3-1)+(5-2)+(4-3)+(5-4)+(6-5)=8$.}
\label{fig:fig3}
\end{SCfigure}

This function $\varphi$ is useful in bounding the relaxation time because it is ``local'' with respect to the Markov chain. That is, suppose $P(x, y) > 0$ and $x \neq y$. The change from $x$ to $y$ could have moved an interior node label to the left, in which case $\varphi(y) = \varphi(x) - 1$. If it moved an interior node label to the right, then $\varphi(y) = \varphi(x) + 1$. If two interior node labels were swapped, then $\varphi(y) = \varphi(x)$. Thus, $(\varphi(x) - \varphi(y))^2 \le 1$. The denominator of Equation \eqref{eqn: relTime} can then be upper-bounded by $1/2$.

To find the variance of the internal tree length, we note that  Var$_{\pi}[\varphi(X)]=$Var$_{\pi^{L}}[\varphi(X)]$ since the internal tree length is independent of the leaf labels. We now re-state the standard coalescent of labeled ranked tree shapes as an urn process. 

Consider an urn containing $n$ balls labeled $\ell_{1},\ldots,\ell_{n}$. At step k, draw two balls without replacement from the urn. Let  $l_1, l_2$ be the labels of the balls drawn, then set $\mathbf{p}_k = (l_1, l_2)^{k}$. Add in a new ball with label $k$. Repeat this process for $k = 1, \dots, n - 1$ until only a single ball labeled $(n-1)$ remains in the urn. The resulting sequence of pairs $\mathbf{p} = (\mathbf{p}_1, \dots, \mathbf{p}_{n-1})$ corresponds to a labeled ranked tree shape  $T \in \T^{L}_{n}$ drawn from the standard coalescent, i.e. $\mathbf{p} \sim \pi^{L}$.

We can now simplify the urn process as follows: start with $n$ white balls in the urn. At each step, draw two balls and add back in a single red ball (representing an interior node of the tree). Let $R_0 = 0$ and $R_k$ be the number of red balls in the urn after $k$ coalescence events. Note that after $k$ mergers (coalescence events), there are $n - k$ total balls left and $R_{n - 1} = 1$.  

The simplified urn process is useful because the internal tree length can be computed by counting the number of red balls at each step and adding them all together (Figure \ref{fig:fig3}), i.e., 
\[
\varphi(\mathbf{p}) = \sum_{k = 1}^{n-2} R_k.
\]

The quantities $\{R_k \}_{k = 1}^{n-2}$ are fairly easy to analyze and have been studied before in various contexts. In \cite{janson2011}, the values are used to study the asymptotic behavior of the external tree length of the Kingman coalescent and it is shown that:
\begin{align*}
& \E[R_k] = \frac{k(n - k)}{n-1}, \\
& Cov(R_k \cdot R_l) = \frac{k(k - 1)(n - l)(n - l -1)}{(n - 1)^2 (n - 2)}, \,\,\,\, k \le l.
\end{align*}
Using this, we get
\[
\E_\pi[\varphi(\mathbf{p})] = \sum_{k = 1}^{n-2} \frac{k(n - k)}{n-1} = \frac{1}{6}(n^2 + n - 6),
\]
Then, we calculate
\begin{align}
&\sum_{k = 1}^{n - 2} \E[R_k^2] = \frac{1}{30}(n^3 + 3n^2 + 2n - 30), \label{eqn: l1}\\
&2 \cdot \sum_{k = 1}^{n-3} \sum_{l = k + 1}^{n - 2} \E[R_k R_l] = \frac{1}{180}(n - 3)(5n^3 + 21 n^2 - 14n - 120). \label{eqn: l2}
\end{align}
The second moment $\E_\pi[\varphi(\mathbf{p})^2]$ is the sum of these two lines \eqref{eqn: l1} + \eqref{eqn: l2}, which gives
\[
\text{Var}_\pi(\varphi(\mathbf{p})) = \E_\pi[\varphi(\mathbf{p})^2] - \E_\pi[\varphi(\mathbf{p})]^2 = \frac{1}{90}n (n + 1)(n - 3).
\]

\begin{theorem} \label{thm:rel}
Let $\tau_{n}, \tau_n^L$ be the relaxation time of the adjacent-swap chains on $\T_n$ and $\T^{L}_{n}$, respectively. Then,
\[
\tau^{L}_{n} \ge \tau_n \ge \frac{\text{Var}_\pi(\varphi)}{\mathcal{E}(\varphi)} \ge \frac{2}{90} n (n + 1)(n - 3).
\]
\end{theorem}

The lower bound on the relaxation time is applicable to labeled and unlabeled ranked tree shapes, however Lemma \ref{lem: lumping2} further allows us to state the first inequality on the left hand side of Theorem \ref{thm:rel}.

\section{Upper bound} \label{sec: coupling}

In this section we prove the upper bounds in Theorem \ref{thm: main} using a coupling argument. The coupling is similar to the one used by Aldous for analyzing the adjacent transpositions chain on $S_n$  \cite{aldous1983random}.

We analyze a lazy version of the chain to make the coupling. That is, at each step, we generate a random coin flip $\theta \sim \text{Bernoulli}(1/2)$. If $\theta = 1$, attempt a move of the chain, and if $\theta = 0$ then make no change.  Let $X_t = (\textbf{p}_1, \dots, \textbf{p}_{n-1})$ and  $Y_t = (\textbf{q}_1, \dots, \textbf{q}_{n-1})$ be the two copies of the chain at time $t$, one started from $x$ and the other from $y$. We will first describe the coupling for the chain on $\T_n$.

\paragraph{Coupling of unlabeled ranked tree shapes.}  We'll define a coupling that jointly matches the internal node labels of the two copies in the order $n-2, n-3, n -4, \dots, 1$. Note that label $n -2$ is already jointly matched in the $n-1$ pairs because it can only occur in the final pair.  Let $N_{t} \in \{1,2,\ldots,n-2\}$ be the maximum label that is not jointly matched, i.e., the maximum element that occurs in different pairs in $X_{t}$ and $Y_{t}$. Let $X_t(a)$ denote the index of the pair in the matching $X_{t}$ that contains $a$, i.e.\ $X_t(a) = i$ if $a \in \textbf{p}_i$ at time $t$. Then, at time $t$, the elements $\geq N_{t}$ match in both $X_{t}$ and $Y_{t}$, i.e. for every $a \geq N_{t}$ $X_{t}(a)=Y_{t}(a)$.\\

For any label $a \in \{1, 2, \dots, n - 2 \}$, the coupling will have two properties:\\

\noindent \textbf{Property 1.} If $a \ge N_t$, then $X_s(a) = Y_s(a)$ for all $s \ge t$.\\

\noindent \textbf{Property 2.} If $a = N_t - 1$ and $X_t(a) < Y_t(a)$, then $X_{t+1}(a) \le Y_{t+1}(a)$ and if $X_t(a) > Y_t(a)$, then $X_{t+1}(a) \ge Y_{t+1}(a)$.  This condition will ensure that $a$ will eventually get matched in the two copies. \\

Define the following quantities:
\begin{itemize}
\item Let $M_{t}$ be the set of all indices $2 \le i \le n -2$ that contain labels $\geq N_{t}$ jointly matched. That is, there is a label $a \geq N_{t}$ such that $X_{t}(a)=Y_{t}(a)=i$.
Note that it is possible to have other matches for labels $< N_t$, but we do not keep track of those matches and breaking those matches does not violate the two properties of the coupling.

\item Let $AM_t = i$ if label $N_t - 1$ is in pairs $\textbf{p}_i$ and $\textbf{q}_{i+1}$, or if it is in pairs $\textbf{p}_{i+1}$ and $\textbf{q}_{i}$. Otherwise, set $AM_t = 0$.
\end{itemize}

At each step, pick an index $1 \le i \le n - 2$ uniformly at random and consider the following cases: 
\begin{enumerate}
\item $i, i+1  \notin M_t$ and $i \neq AM_t$. There are no joint matchings of labels $\geq N_{t}$ in pairs $i$ and $i+1$ and swapping two elements in both copies will not match label $N_{t}-1$. In this case, propose a swap in $X_{t}$ and $Y_{t}$ according to their (lazy) marginal dynamics.
 

\item $i \in M_{t}$ and/or $i+1 \in M_t$, and $i \neq AM_t$. There is at least one joint match of labels $\geq N_{t}$ in pairs $i$ or $i+1$ and swapping two elements in both copies will not create a new joint match of $N_{t}-1$. To preserve Property 1 it is necessary to perform the \textit{same} move (or no move at all) on both chains. Toss a single coin $\theta$ to determine whether a move will be proposed on both copies or not. \\
(a) If $\theta=1$ and the pairs at $i$ and $i + 1$ are of the type:
\begin{align*}
&X: (a, b)^i, (c, d)^{i+1} \\
&Y: (a, e)^i, (f, g)^{i+1},
\end{align*}
With the possibility of $b=e$. Then, if a matched label e.g. $a$ is chosen to be swapped in chain $X$, it must also be chosen for $Y$. Now, label $a<i$ since it is a label that merges at $i$; and $a \geq N_{t}$ since it is a label that was jointly matched before. Then $c,d,f,$ or $g$ cannot be label $i$ since $i$ is already a match. This implies that there are no constraints on $c,d,f,g$ and the  marginal probability of swapping $a$ would be the same in each chain. 

(b) If $\theta=1$ and the pairs at $i$ and $i + 1$ are of the type:
\begin{align*}
&X: (a, b)^i, (c, d)^{i+1} \\
&Y: (e, f)^i, (g, d)^{i+1},
\end{align*}
With the possibility of $d=i$, or 
\begin{align*}
&X: (a, b)^i, (i, d)^{i+1} \\
&Y: (e, f)^i, (i, d)^{i+1},
\end{align*}
Then, if a matched label e.g. $d$ is chosen to be swapped in chain $X$, it must also be swapped in $Y$. 

\item $i,i+1 \notin M_t$ and $i = AM_t$. There are no joint matchings of labels $\geq N_{t}$ in pairs $i$ and  $i+1$, and label $N_{t}-1$ is either in $\textbf{p}_{i}$ and $\textbf{q}_{i+1}$ or  in $\textbf{p}_{i+1}$ and $\textbf{q}_{i}$. In order to preserve property 2,  simply set $\theta_Y = 1 - \theta_X$. This ensures that label $a = N_t - 1$ will only possibly move in one of the chains.

\item $i \in M_t$ and/or $i+1 \in M_{t}$, and $i = AM_t$. There is one joint match in pair $i$ and/or $i+1$, and label $N_{t}-1$ is either in $\textbf{p}_{i}$ and $\textbf{q}_{i+1}$ or in $\textbf{p}_{i+1}$ and $\textbf{q}_{i}$. To preserve Properties 1 and 2 
while keeping the correct marginal transition probabilities for each chain we define the joint transitions as follows. 

(a) Suppose $c = N_t - 1$ and the pairs at indices $i, i + 1$ are of the type:
\begin{align*}
&X: (a, b)^i, (c, d)^{i+1} \\
&Y: (a, c)^i, (f, g)^{i+1}.
\end{align*}
By the same argument of case 2, $c,d, g,f \neq i$. Table \ref{tab:table1} defines the joint proposal probabilities that preserve Properties $1$ and $2$ and correct marginal transition probabilities.


\begin{table}[h]
\caption{\label{tab:table1} Coupling transition probabilities in case 4(a)}
\begin{center}
\begin{tabular}{c|c|r}
Move in $X$ & Move in $Y$ & Probability\\
\hline
No change & No change & $3/8$\\
No change & $c \leftrightarrow f$ & $1/16$ \\
No change & $c \leftrightarrow g$ & $1/16$ \\
$b \leftrightarrow c$ & No change & $1/8$ \\
$b \leftrightarrow d$ & $c \leftrightarrow f$  & $1/16$ \\
$b \leftrightarrow d$ & $c \leftrightarrow g$  & $1/16$ \\
$a \leftrightarrow c$ & $a \leftrightarrow f$  & $1/16$ \\
$a \leftrightarrow c$ & $a \leftrightarrow g$  & $1/16$ \\
$a \leftrightarrow d$ & $a \leftrightarrow f$  & $1/16$ \\
$a \leftrightarrow d$ & $a \leftrightarrow g$  & $1/16$ \\
\end{tabular}
\end{center}
\end{table}

(b) Suppose now that $c=N_{t}-1$ and that pairs at indices $i$ and $i+1$ are of the type:
\begin{align*}
&X: (a, b)^i, (c, i)^{i+1} \\
&Y: (a, c)^i, (f, i)^{i+1}.
\end{align*}
Table \ref{tab:table2} defines the joint proposal probabilities that preserve Properties 1 and 2 with correct marginal transition probabilities. 

\begin{table}[h]
\caption{\label{tab:table2} Coupling transition probabilities in case 4(b)}
\begin{center}
\begin{tabular}{c|c|r}
Move in $X$ & Move in $Y$ & Probability\\
\hline
No change & No change & $5/8$\\
No change & $c \leftrightarrow f$ & $1/8$ \\
$b \leftrightarrow c$ & No change & $1/8$ \\
$a \leftrightarrow c$ & $a \leftrightarrow f$  & $1/8$ \\
\end{tabular}
\end{center}
\end{table}

(c) Suppose now that $c=N_{t}-1$ and that pairs at indices $i$ and $i+1$ are of the type:
\begin{align*}
&X: (a, c)^i, (b, i)^{i+1} \\
&Y: (d, e)^i, (c, i)^{i+1}.
\end{align*}
Table \ref{tab:table3} defines the joint proposal probabilities that preserve Properties 1 and 2 with correct marginal transition probabilities. 
\begin{table}[h]
\caption{\label{tab:table3} Coupling transition probability in case 4(c).}
\begin{center}
\begin{tabular}{c|c|r}
Move in $X$ & Move in $Y$ & Probability\\
\hline
No change & No change & $5/8$\\
No change & $e \leftrightarrow c$ & $1/8$ \\
$b \leftrightarrow c$ & No change & $1/8$ \\
$a \leftrightarrow b$ & $d \leftrightarrow c$  & $1/8$ \\
\end{tabular}
\end{center}
\end{table}

\end{enumerate}

\paragraph{Coupling of labeled ranked tree shapes.} The coupling for $\T_n^L$ can proceed exactly as the coupling for $\T_n$ until every interior-node label is matched, say at time $T$. After that point, the leaf labels can be matched in any order. We extend the set $M_t$ to contain the indices $1 \le i \le n -2$ of any label $a \in \{\ell_{1},\ldots,\ell_{n}\}$ such that $X_t(a) = Y_t(a)$ and $AM_t$ to be the set of indices $i$ such that any label $a\in \{\ell_{1},\ldots,\ell_{n}\}$ is in pair $\textbf{p}_i$ and $\textbf{q}_{i+1}$ or $\textbf{p}_{i+1}$ and $\textbf{q}_i$. The transition probabilities defined above work with these sets $M_t$, $AM_t$ and have the following properties, for $t \ge T$ and label $a$:

\begin{enumerate}
\item \textbf{Property 1} If $X_t(a) = Y_t(a)$, then $X_s(a) = Y_s(a)$ for all $s \ge t$. 
\item \textbf{Property 2} If $X_t(a) < Y_t(a)$ then $X_s(a) \le Y_s(a)$ for all $s \ge t$. If $X_t(a) > Y_t(a)$ then $X_s(a) \ge Y_s(a)$ for all $s \ge t$.
\end{enumerate}

\subsection{Coupling time}

%
For the chain on unlabeled ranked tree shapes $\T_n$, the time to couple is $T = T_{n-3} + T_{n-4} + \dots + T_1$, where $T_a$ is the time it takes to match interior node label $a$, after the labels $a +1, \dots, n - 2$ are already matched. Property 2 is crucial to study $T_{a}$. Suppose that at time $t = T_{n-3} + \dots + T_{a + 1}$ when label $a + 1$ is matched we have $X_t(a) < Y_t(a)$. Let $S_a$ be the time it takes for $Y_t$ to hit the left boundary of it's range, i.e.\  $Y_t(a) = a + 1$. Then necessarily at this time, by Property 2, $X_t(a) = Y_t(a)$.

Note that $X_t(a) \in \{a+1, \dots, n - 2\}$. 
If $a$ is not at the left boundary, the probability of moving $a$ to the left is:
\[
\P(X_{t+1}(a) = X_t(a) - 1 \mid X_t \neq a + 1 ) = \frac{1}{n - 2} \cdot \frac{1}{2} \cdot \frac{1}{2} = \frac{1}{4(n-2)}.
\]
The probability that $a$ moves to the right will depend on the exact configuration and whether or not there is a constraint, e.g.\ in the situation
\begin{align*}
& (a, b)^i, (c, i)^{i+1},
\end{align*}
in which the probability $a$ moves to the right is $1/8(n -3)$. Thus, we can bound
\[
\frac{1}{8(n - 2)} \le \P(X_{t+1}(a) = X_t(a) + 1 \mid X_t \neq n - 2 ) \le \frac{1}{4(n-2)}.
\]


The following is an easy result about the hitting time of a symmetric random walk on a line.
\begin{lemma} \label{lem: hittingTime}
Let $(Z_t)_{t \ge 0}$ be a random walk on the line $\{1, 2, 3, \dots, m\}$ with the transitions
\begin{align*}
&P(i, i+1) = p \,\,\,\,\,\,\,\,\,\,\, i \neq m \\
&P(i, i-1) = p \,\,\,\,\,\,\,\,\,\,\, i \neq 1 \\
&P(i, i) = 1 - 2p \,\,\,\,\,\,\,\,\,\,\, i \neq 1, m \\
&P(m, m) = 1 - p\\
&P(1, 1) = 1
\end{align*}
for some $0 < p \le 1/2$. Suppose $Z_0 = m$ and let $T_m = \inf \{t > 0: Z_t =1 \}$. Then $\E[T_m] = (1/2p)\cdot m(m - 1)$.
\end{lemma}

\begin{proof}
We can prove the result for $p = 1/2$, then scale. Let $a_x$ be the expected first time of hitting $1$ if $Z_0 = x$. Note that the following identities are satisfied:
\begin{align*}
&a_1 = 0 \\
&a_x = 1 + \frac{1}{2}(a_{x + 1} + a_{x - 1}) \,\,\,\,\,\, x \in \{2, 3, \dots, m - 1 \} \\
&a_m = 1 + \frac{1}{2} a_{m-1} + \frac{1}{2} a_m.
\end{align*}
This is a second-order recursion, with solution $a_x = x(2m - x + 1) - 2m$. So this means $a_m = m(m - 1)$. When $p < 1/2$ the probability of moving is $2p < 1$, so the expected time to hit $1$ started from $m$ is $(1/2p)m(m - 1)$.
\end{proof}


Let $Z_t$ be a random walk on the line $a+1, \dots, n - 2$ with $p = 1/4(n-2)$, defined as in Lemma \ref{lem: hittingTime}. We can couple $Z_t$ with $Y_t(a)$ so that $Y_t(a) \le Z_t$ almost surely for all $t \ge 0$, because $Z_t$ is moving to the right with probability $\ge$ $Y_t(a)$. In conclusion, 
\[
\E[T_a] \le 2(n-2) (n - 2 - a)^2
\]

and thus the total coupling time is

\[
\E[\tau_{couple}] \le \sum_{a = 1}^{n - 3} \E[T_a] \le \sum_{a = 1}^{n - 3} 2(n-2) (n - 2 - a)^2 =   \frac{1}{3}(n-1)(n - 2)(n-3) (2n-5) \leq \frac{2}{3} n^4.
\]
This together with Theorem \ref{thm: couplingThm} gives the upper bound in Theorem \ref{thm: main} for $\T_n$:
\[
t_n \le 4 \cdot \E[\tau_{couple}] \leq \frac{8}{3}n^4.
\]

\paragraph{Leaf-Labeled Trees}
The coupling time for leaf-labeled trees can be written $\tau_{couple}^L = T + T^{leaves}$, where $T$ is the time from the previous section for the interior labels to match. For the time $T^{leaves}$ for the leaves to match, note that time is saved because the leaves are allowed to match in any order. Moreover, leaf labels can be moved without constraints.  In analogy with the result from Example 4.10 in \cite{aldous1983random} for adjacent transpositions on $S_n$, this would take time of order  $n^3 \log(n)$. Therefore, $\E[\tau_{couple}^{L}]\leq  \frac{2}{3} n^4 + C_{4}n^{3}\log(n)$.

\section{Discussion} \label{sec: discussion}

The representation of ranked tree shapes as ordered matchings can be used to define many other Markov chains in analogy to well-studied chains on $S_n$. For example, in addition to the adjacent-swap, another natural chain would be a random-swap Markov chain: pick two pairs uniformly at random, within each pair pick a label, and swap the two elements if it is allowed. The internal tree length function defined in Section \ref{sec: lowerBound} gives a lower bound of $n$ for this chain, because a single move could change the function by at most $n$, and thus the Dirichlet form can be bounded by $n^2$. A naive coupling argument would give an upper bound of order $n^3$. 
The focus of this paper was on the adjacent-swap chain because it is a local-move chain which could be more useful for Metropolis algorithms in applications. 

The upper bound on the mixing time relies on a coupling that matches one label at a time in a specific order giving a sub-optimal upper bound. We conjecture that the mixing time of the adjacent-swap chain on unlabeled ranked tree shapes is of the order of $n^{3}\log(n)$ as it is in the case of adjacent transpositions on $S_{n}$. Figure \ref{fig:tvdist} shows the total variation distance between the adjacent-swap chain on unlabeled ranked tree shapes and the Tajima stationary distribution for trees with $n<10$ leaves. Due to the large size of the state space, this calculation could not be extended for larger trees in order to inform about the presence of a cutoff phenomena.

\begin{figure}
\includegraphics[width=0.5\textwidth]{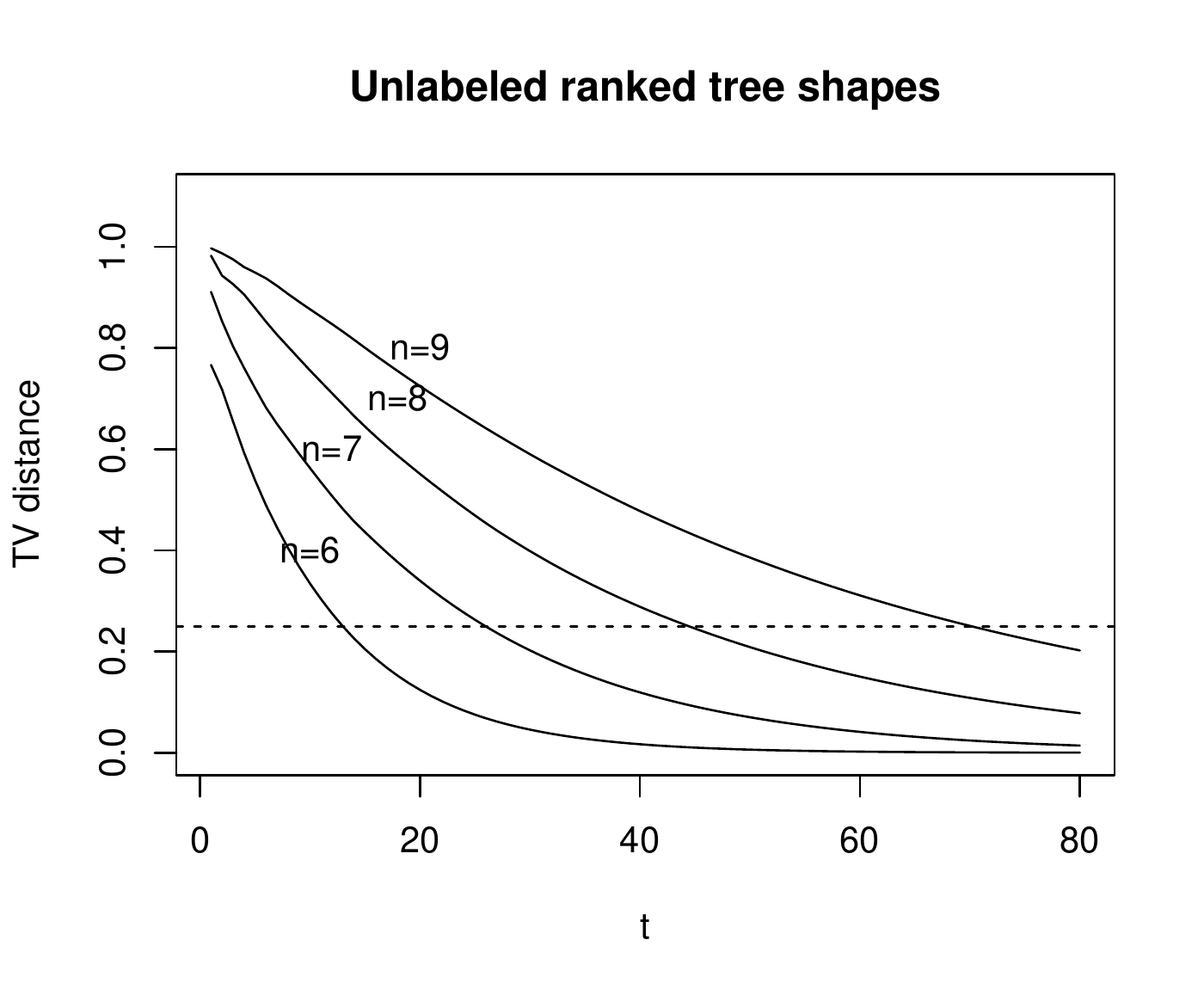}
\centering
\caption{\textbf{Total variation distance for adjacent-swap chain on unlabeled ranked tree shapes}. Each curve depicts the total variation distance between the probability of the chain at time $t$ and the Tajima stationary distribution for trees with $n=6,7,8$ and $9$ leaves.}
\label{fig:tvdist}
\end{figure}

We also note another connection between trees and permutations: Unlabeled ranked tree shapes are in bijection with alternating permutations \citep{donaghey1975alternating,richard1999enumerative} and hence some results concerning Markov chains on alternating permutations could be used to analyze the mixing of Markov chains on unlabeled ranked trees and vice versa. However, we are not aware of mixing time results of Markov chains on the space alternating permutations \citep{diaconis2013random}. 

Finally, in this work we analyzed Markov chains on trees whose stationary distributions are coalescent distributions. These coalescent distributions are used as \textit{prior} distributions over the set of ancestral relationships of samples of DNA \citep{Palacios2019}. In practice, these Markov chains are used to approximate the posterior distribution, and hence, the state space is more restricted and subject of future exploration.

\paragraph{Acknowledgments.} MS is supported by a National Defense Science \& Engineering Graduate fellowship. JAP is supported by  National Institutes of Health Grant R01-GM-131404 and the Alfred P. Sloan Foundation.
\bibliographystyle{abbrv}
\bibliography{citations}

\begin{thebibliography}{10}

\bibitem{aldous1983random}
D.~Aldous.
\newblock Random walks on finite groups and rapidly mixing {M}arkov chains.
\newblock In {\em S{\'e}minaire de Probabilit{\'e}s XVII 1981/82}, pages
  243--297. Springer, 1983.

\bibitem{aldous2000mixing}
D.~J. Aldous.
\newblock Mixing time for a {M}arkov chain on cladograms.
\newblock {\em Combinatorics, Probability and Computing}, 9(3):191--204, 2000.

\bibitem{diaconis2002random}
P.~Diaconis, S.~Holmes, et~al.
\newblock Random walks on trees and matchings.
\newblock {\em Electronic Journal of Probability}, 7, 2002.

\bibitem{diaconis2013random}
P.~Diaconis and P.~M. Wood.
\newblock Random doubly stochastic tridiagonal matrices.
\newblock {\em Random Structures \& Algorithms}, 42(4):403--437, 2013.

\bibitem{diaconis1998matchings}
P.~W. Diaconis and S.~P. Holmes.
\newblock Matchings and phylogenetic trees.
\newblock {\em Proceedings of the National Academy of Sciences},
  95(25):14600--14602, 1998.

\bibitem{donaghey1975alternating}
R.~Donaghey.
\newblock Alternating permutations and binary increasing trees.
\newblock {\em Journal of Combinatorial Theory, Series A}, 18(2):141--148,
  1975.

\bibitem{drummond2012bayesian}
A.~Drummond, M.~Suchard, D.~Xie, and A.~Rambaut.
\newblock Bayesian phylogenetics with {BEAU}ti and the {BEAST} 1.7.
\newblock {\em Molecular Biology and Evolution}, 29:1969--1973, 2012.

\bibitem{durrett2003shuffling}
R.~Durrett.
\newblock Shuffling chromosomes.
\newblock {\em Journal of Theoretical Probability}, 16(3):725--750, 2003.

\bibitem{felsenstein2004inferring}
J.~Felsenstein.
\newblock {\em Inferring phylogenies}, volume~2.
\newblock Sinauer associates Sunderland, MA, 2004.

\bibitem{Frost2013}
S.~D.~W. Frost and E.~M. Volz.
\newblock Modelling tree shape and structure in viral phylodynamics.
\newblock {\em Philosophical Transactions of the Royal Society B: Biological
  Sciences}, 368(1614):20120208, 2013.

\bibitem{janson2011}
S.~Janson and G.~Kersting.
\newblock On the total external length of the {K}ingman coalescent.
\newblock {\em Electron. J. Probab.}, 16:2203--2218, 2011.

\bibitem{Kingman:1982uj}
J.~Kingman.
\newblock {The coalescent}.
\newblock {\em Stochastic Processes and their Applications}, 13(3):235--248,
  1982.

\bibitem{Kirkpatrick1993}
M.~Kirkpatrick and M.~Slatkin.
\newblock Searching for evolutionary patterns in the shape of a phylogenetic
  tree.
\newblock {\em Evolution}, 47(4):1171--1181, 1993.

\bibitem{kuznetsov1994increasing}
A.~G. Kuznetsov, I.~M. Pak, and A.~E. Postnikov.
\newblock Increasing trees and alternating permutations.
\newblock {\em Russian Mathematical Surveys}, 49(6):79, 1994.

\bibitem{lacoin2016mixing}
H.~Lacoin et~al.
\newblock Mixing time and cutoff for the adjacent transposition shuffle and the
  simple exclusion.
\newblock {\em The Annals of Probability}, 44(2):1426--1487, 2016.

\bibitem{LevinPeresWilmer2006}
D.~A. Levin and Y.~Peres.
\newblock {\em Markov chains and mixing times}, volume 107.
\newblock American Mathematical Society, 2017.

\bibitem{Maliet2018}
O.~Maliet, F.~Gascuel, and A.~Lambert.
\newblock Ranked tree shapes, nonrandom extinctions, and the loss of
  phylogenetic diversity.
\newblock {\em Systematic Biology}, 67(6):1025--1040, 04 2018.

\bibitem{misra2011optimization}
N.~Misra, G.~Blelloch, R.~Ravi, and R.~Schwartz.
\newblock An optimization-based sampling scheme for phylogenetic trees.
\newblock In {\em International Conference on Research in Computational
  Molecular Biology}, pages 252--266. Springer, 2011.

\bibitem{mossel2006limitations}
E.~Mossel, E.~Vigoda, et~al.
\newblock Limitations of markov chain monte carlo algorithms for bayesian
  inference of phylogeny.
\newblock {\em The Annals of Applied Probability}, 16(4):2215--2234, 2006.

\bibitem{Palacios2019}
J.~A. Palacios, A.~V{\'e}ber, L.~Cappello, Z.~Wang, J.~Wakeley, and
  S.~Ramachandran.
\newblock Bayesian estimation of population size changes by sampling
  {T}ajima{\textquoteright}s trees.
\newblock {\em Genetics}, 213(3):967--986, 2019.

\bibitem{richard1999enumerative}
P.~S. Richard.
\newblock Enumerative combinatorics.
\newblock {\em Volume I}, 1999.

\bibitem{veber}
R.~Sainudiin, T.~Stadler, and A.~V\'{e}ber.
\newblock Finding the best resolution for the {Kingman-Tajima} coalescent:
  theory and applications.
\newblock {\em Journal of Mathematical Biology}, 70(6):1207–1247, 2015.

\bibitem{schweinsberg2002n2}
J.~Schweinsberg.
\newblock An ${O}(n^{2})$ bound for the relaxation time of a {M}arkov chain on
  cladograms.
\newblock {\em Random Structures \& Algorithms}, 20(1):59--70, 2002.

\bibitem{SHK2014}
D.~Spade, R.~Herbei, and L.~Kubatko.
\newblock A note on the relaxation time of two {M}arkov chains on rooted
  phylogenetic tree spaces.
\newblock {\em Statistics \& Probability Letters}, 84:247–252, 01 2014.

\bibitem{vstefankovivc2011fast}
D.~{\v{S}}tefankovi{\v{c}} and E.~Vigoda.
\newblock Fast convergence of {M}arkov chain {M}onte {C}arlo algorithms for
  phylogenetic reconstruction with homogeneous data on closely related species.
\newblock {\em SIAM Journal on Discrete Mathematics}, 25(3):1194--1211, 2011.

\bibitem{wakeley_coalescent_2008}
J.~Wakeley.
\newblock {\em Coalescent Theory: An Introduction}.
\newblock Roberts \& Company Publishers, June 2008.

\bibitem{wilson2004mixing}
D.~B. Wilson et~al.
\newblock Mixing times of lozenge tiling and card shuffling {M}arkov chains.
\newblock {\em The Annals of Applied Probability}, 14(1):274--325, 2004.

\end{thebibliography}

\end{document}